\documentclass[12pt]{amsart}

\usepackage{amscd, amssymb, amsthm, amsxtra, array}
\usepackage{mathtools}
\usepackage{pdfsync}
\usepackage{url}
\usepackage{orcidlink}

\evensidemargin=\oddsidemargin 

\newcommand{\la}{\langle}
\newcommand{\ra}{\rangle}

\newcommand\op{{\oplus}}
\newcommand{\smallfrac}[2]{\textstyle{\frac{#1}{#2}}  }

\newcommand{\bfZ}{\ensuremath{\mathbf Z} }

\newcommand{\bfB}{\ensuremath{\mathbf B} }

\newcommand{\bfD}{\ensuremath{\mathbf D} }
  
\newcommand{\bfH}{\ensuremath{\mathbf H} }

\newcommand\boldR{\ensuremath{\mathbb R}}
\newcommand\R{\mathbb R}
\newcommand\boldZ{{\mathbb Z}}

\newcommand{\calB}{\ensuremath{\mathcal{B}} }
\newcommand{\calC}{\ensuremath{\mathcal{C}} }

\newcommand\Fix{\operatorname{Fix}}

\newcommand{\trace}{\operatorname{trace}}
\newcommand{\Id}{\operatorname{Id}}
\newcommand{\myspan}{\operatorname{span}}

\newcommand{\fraka}{\ensuremath{\mathfrak{a}} }

\newcommand{\frakg}{\ensuremath{\mathfrak{g}} }
\newcommand{\frakh}{\ensuremath{\mathfrak{h}} }

\newcommand{\frakk}{\ensuremath{\mathfrak{k}} }

\newcommand{\frakn}{\ensuremath{\mathfrak{n}} }

\newcommand{\frakq}{\ensuremath{\mathfrak{q}} }

\newcommand{\fraks}{\ensuremath{\mathfrak{s}} }

\newcommand\fraksl{{\mathfrak{sl}}}

\renewcommand\t{{\mathfrak{t}}}

\renewcommand\a{{\mathfrak{a}}}
\newcommand\g{{\mathfrak{g}}}

\newcommand\n{{\mathfrak{n}}}

\newcommand\s{{\mathfrak{s}}}

\newcommand\ip{{\langle\cdot {,} \cdot\rangle}}

\newcommand\tr{\operatorname{trace}}

\newcommand\diag{\operatorname{diag}}

\newcommand\ad{\operatorname{ad}}
\newcommand\End{\operatorname{End}}

\newcommand\ric{\operatorname{ric}}

\newcommand{\Ric}{\operatorname{Ric}}

\theoremstyle{plain}
\newtheorem{theorem}{Theorem} 
\newtheorem*{main}{Main Theorem}
\newtheorem*{def-2-1}{Definition 2.1}
\newtheorem{prop}[theorem]{Proposition}
\newtheorem{lemma}[theorem]{Lemma}  
\newtheorem{coro}[theorem]{Corollary}

\theoremstyle{definition} 
\newtheorem{definition}[theorem]{Definition}
\newtheorem{defn}[theorem]{Definition}

\theoremstyle{remark}
\newtheorem{remark}[theorem]{Remark}

\numberwithin{equation}{section}
\numberwithin{theorem}{section}

\begin{document}
\title[Attached Submanifolds Beyond Symmetric Spaces]{Attached Submanifolds Beyond Symmetric Spaces}

\thanks{} 

\author[M. M. Kerr]{Megan M. Kerr \orcidlink{https://orcid.org/0000-0003-2841-1752}}
\address{Department of Mathematics, Wellesley College, 106 Central St.,
Wellesley, MA 02481} \email{mkerr@wellesley.edu}
\author[T. L. Payne]{Tracy L. Payne  \orcidlink{https://orcid.org/0000-0002-0218-3352}}
\address{Department of Mathematics and Statistics,  Idaho State University, 
Pocatello, ID 83209-8085} \email{tracypayne@isu.edu}
\date{\today}

\begin{abstract} We study submanifold geometry in the presence
of symmetry, focusing on  submanifolds of solvmanifolds with an
unusual property relative to Ricci curvature.   We generalize
work  of H. Tamaru  \cite{tamaru-11} in which he explores the
geometry of submanifolds of symmetric spaces of noncompact type
constructed from parabolic subgroups of  the isometry group.  
He calls these {\em attached submanifolds.}
The Ricci curvatures of attached submanifolds coincide with the
restrictions of the Ricci curvatures of  ambient symmetric
spaces. 

We broaden Tamaru's construction by  weakening the hypotheses
on the  ambient space, allowing  a pseudo-Riemannian scalar
product, and defining attached submanifolds in terms of root
spaces.  
 We demonstrate that in this setting, the Ricci curvature
 restriction property for attached submanifolds holds if and
 only if the submanifold satisfies an algebraic criterion that
 we call the Jacobi Star Condition.
Like  attached submanifolds of symmetric spaces, our attached
submanifolds are minimal, and are only totally geodesic under
hypotheses analogous to hypotheses in the symmetric space case.

Finally, we give an example of a solvmanifold that has  an attached submanifold  and is not a symmetric space, 
demonstrating that attached submanifolds are not unique to symmetric spaces. 
\end{abstract}
\maketitle

\section{Introduction}\label{introduction}

\subsection{Context and motivation.}
 In a sequence of papers, Hiroshi 
Tamaru studied submanifolds of symmetric spaces of noncompact type defined by parabolic subalgebras in corresponding semisimple Lie algebras \cite{tamaru-08, tamaru-11, tamaru-11b}. He calls these {\em attached} submanifolds.   The special properties of parabolic subalgebras translate to unique geometric features of the attached submanifolds.

Of particular interest, Tamaru shows  that  the Ricci curvature tensor for the  attached submanifold coincides with the restriction of the Ricci curvature tensor of the  ambient manifold. 
 As a corollary, because symmetric spaces are Einstein,  their attached  submanifolds
 inherit the constant Ricci curvature of the larger space, yielding many new examples of
 Einstein homogeneous spaces.  In fact, these manifolds are  solvmanifolds, and they
 include examples for which the underlying solvable Lie group has higher step. 
 Another notable property is minimality: Tamaru shows that attached submanifolds are
 always minimal. 
In addition, although they are not always totally geodesic, Tamaru presents necessary
and sufficient conditions for them to be so.
In this work,
we ask,  is  the property of restricted ambient Ricci curvature agreeing with the
intrinsic Ricci curvature of a submanifold unique to attached submanifolds of symmetric
spaces, or does it hold more generally?  Under what more general hypotheses  will
Tamaru's approach carry through?  Throughout his analysis, Tamaru uses the  machinery of
semisimple Lie groups, such as the Killing form, root spaces and the Weyl group.
We carefully analyze the hypotheses and proofs in Tamaru's work, and we generalize his
results by  broadening the context beyond the symmetric spaces, thus weakening
hypotheses on both the ambient space and the submanifold. We discover necessary
geometric hypotheses that allow us to preserve Tamaru's main results.

After generalizing Tamaru's theorems, we ask whether there are 
solvable metric Lie algebras that meet our more general hypotheses and do not meet 
Tamaru's; i.e., do not correspond to symmetric spaces of noncompact type.  We give an
example showing our hypotheses are truly more general, because they are  satisfied by a
solvmanifold that is not a symmetric space. The example we present  is part of a
naturally defined family.   The computations in our example are 
technical and lengthy.  To work more generally and construct  other 
examples, one should  use the structure of 
Kac-Moody Lie algebras, which we have suppressed in our exposition 
to improve readability.
We were motivated to use Kac-Moody algebras because the finite-dimensional solvable
metric Lie algebras  that they define are of strong Iwasawa type, and crucially,  in a
Kac-Moody algebra, the maps $\ad_X^\ast$ are derivations, and this basic condition 
ensures that the Jacobi Star Condition holds for our examples.  

\subsection{Tamaru's results on attached submanifolds of symmetric spaces of noncompact type}

 Let  $(M,g)$ be a symmetric space of noncompact type  and let  $\frakg$ be the Lie
 algebra of the isometry group of $(M,g).$  Let $\frakq$ be a parabolic subalgebra of
 $\frakg$.  Tamaru defined an  orbit of $\exp(\frakq)$ in $M$, endowed with the induced
 metric,   to be an {\em attached submanifold.}  

 An attached submanifold $(S',g')$ of  $(M,g)$  has special properties relative to the ambient  symmetric space.   The main theorem of \cite{tamaru-11} states that the Ricci  form for $(S',g')$  is the 
restriction of the Ricci form of the ambient space  $(S,g)$.     Because the symmetric space is Einstein,  every attached submanifold  is Einstein.  
Furthermore, all attached submanifolds are minimal.  They are not always totally geodesic, and Tamaru presents necessary and sufficient conditions for the attached submanifold to be totally geodesic. 

\subsection{Main theoretical results}\label{Main results}

We are able to generalize  many 
of the results from Tamaru's paper 
\cite{tamaru-11}, often closely following his proofs.

Throughout this work, we work in the setting of metric Lie algebras. 
A solvable Lie group $S$ endowed with a left-invariant Riemannian or pseudo-Riemannian metric $g$  may be identified with a Lie algebra $(\fraks,\ip)$ equipped with a scalar  product, and all local geometric quantities for the solvmanifold $(S,g)$ may be computed via the metric Lie algebra $(\fraks,\ip).$ 

\subsubsection{Main Theorem}

We consider solvable pseudo-Riemannian metric Lie algebras of strong Iwasawa type $(\fraks,\ip)$ as defined in Definition \ref{defn: Iwasawa}.
In Definition \ref{defn: attached}, we define an {\em attached metric subalgebra}
$(\fraks',\ip')$ defined by a subset $\Lambda'$ of roots coming from a root space
decomposition. 
   We determine in particular when the restriction of
the Ricci curvature for $(\fraks,\ip)$ to $\fraks'$ agrees with the intrinsic Ricci
curvature for $(\fraks',\ip')$. 
\begin{definition}\label{jacobi star} Let $(\fraks,\ip)$ be a solvable pseudo-Riemannian
metric  Lie algebra of strong Iwasawa type, with $\Lambda \subseteq \fraka^\ast$,
and let  $(\s', \ip')$ be an attached subalgebra defined by  $\Lambda'$ as in Definition
\ref{defn: attached}.  
We say that $(\fraks',\ip')$ satisfies the {\em Jacobi Star Condition}  if there is a
compatible orthonormal basis $\{E_j'\} \cup \{E_j^\perp\}$ 
for $\n = \n' \oplus \n_0$ such that
\begin{equation}\label{eq:jacobi-star}
\smallfrac{1}{2} \sum _j \epsilon_j  [(\ad_{E_j^\perp})^{\ast, \frakn} ,
\ad_{E_j^\perp}
]
(X) =
\ad_{ \sum  \epsilon_j (\ad_{E_j^\perp})^{\ast, \fraks} E_j^\perp} (X) \end{equation}
for all $X$ in $\frakn'$.  
\end{definition}
This definition may seem technical.  A slightly more restrictive condition which implies the Jacobi Star Condition is that $\ad_X^\ast$ be a derivation for all $X.$ Indeed, this is the only additional property of symmetric spaces that we use in Corollary \ref{cor: hiroshi},  where we show that parabolic subalgebras of symmetric spaces satisfy the Jacobi Star Condition.  

Our main result is that the restriction of the Ricci
   endomorphism of a solvable metric Lie algebra   $(\fraks,\ip)$ of strong Iwasawa type
   to $\fraks'$ coincides with the Ricci endomorphism for
   $(\fraks',\ip')$ if and only if the Jacobi Star Condition holds.  
   \begin{main}\label{thm: Ricci equal}
   Let $(\fraks = \fraka \oplus \frakn,\ip)$ be a solvable pseudo-Riemannian metric  Lie algebra
   of strong  Iwasawa type. 
   Let $(\fraks' = \fraka' \oplus \frakn', \ip')$ be the attached metric subalgebra defined by a subset $\Lambda'$ of $\fraka^\ast$.
Let $H$ and $H'$ denote the mean curvature vectors for $(\s,\ip)$ and $(\s',\ip')$, respectively.  Let $\Ric^\s$, $\Ric^{\s'}$, $\Ric^\n$, and $\Ric^{\n'}$ denote the Ricci endomorphisms of the corresponding metric Lie algebras. 
   Then the following are equivalent:
   \begin{enumerate}
       \item 
 ${\displaystyle \Ric^{\s}(X) = \Ric^{\s'}(X)}$   for any $X \in \s'$.
\item $\Ric^\n (X) - \Ric^{\n'} (X) =  [H - H', X]$ for any $X \in \n'$.
  \item The Jacobi Star Condition  holds. 
\end{enumerate}
\end{main}

\subsubsection{Minimality}

Furthermore, we show in Theorem \ref{thm: minimal} that attached
submanifolds are minimal submanifolds, and in Proposition \ref{prop: tot-geodesic},
we characterize when they are totally geodesic. These
results extend analogous statements for symmetric spaces of
noncompact type in \cite{tamaru-11}.

\subsection{Solvable Metric Lie algebras defined by  Kac-Moody algebras}
 We will present   an example of a solvable metric Lie algebra 
that meets our hypotheses but does not correspond to a symmetric space.
 This example comes from the affine untwisted Kac-Moody
algebra $\frakg$ defined by $\fraksl_3(\boldR).$ 
The infinite-dimensional Lie algebra $\frakg$ has a triangular decomposition
$\frakg = \frakn_- + \frakh + \frakn_+.$  We take the upper
triangular part $\frakn_+$ and truncate it to obtain a
finite-dimensional nilpotent Lie algebra $\frakn$.  We extend $\frakn$  by an
abelian  subalgebra of $\frakh$ to get a  12-dimensional, 6-step 
solvable Lie algebra $\fraks = \fraka \ltimes \frakn$ of strong Iwasawa type.  We use a natural invariant bilinear form from the Kac-Moody algebra $\frakg$ to define an inner product on $\fraks.$  
 
This   construction, the truncation of a  triangular part of a 
 Kac-Moody algebra,  works  more generally.  By varying the Kac-Moody Lie algebra and varying the ideal used in the truncation, infinitely many examples may be constructed.  We conjecture that in general many of the symmetries of the Kac-Moody Lie algebra will be preserved in the finite-dimensional nilpotent quotient.

\subsection{Overview}   In Section \ref{section: preliminaries} we review the algebraic and geometric properties of  solvmanifolds with a pseudo-Riemannian metric. In Section \ref{section: attached} we define an attached subalgebra 
and we describe its properties. In Section \ref{section: geometric results} we prove our main theorem 
on attached subalgebras. We also prove that attached submanifolds are minimal and prove necessary and sufficient conditions for them to be  totally geodesic. In Section \ref{section: symmetric spaces} we show that Tamaru's attached subalgebras arising from noncompact symmetric spaces are a special case of our construction. Finally, in Section \ref{section: example}, we give an example of a solvmanifold that is not a symmetric space with an attached submanifold.

\section{Preliminaries}\label{section: preliminaries}

\subsection{Solvable metric Lie algebras}\label{sec: solvable}
A {\em (pseudo-Riemannian) metric Lie algebra} is a Lie algebra
endowed with a nondegenerate
symmetric bilinear form, the scalar product $\ip$.
Let $(\fraks,\ip)$ be a solvable pseudo-Riemannian Lie algebra.
A vector $X$ in $\fraks$ is a {\em unit vector} with respect to the
scalar product  if  $\la X,X\ra  \in \{-1,1\}$.
A basis $\{E_i\}$ of pairwise orthogonal  unit vectors is called {\em
  orthonormal}.
For such a basis, we use   $\epsilon_i \in \{-1,1\}$ to denote $\la E_i, E_i \ra$.
We focus on solvable pseudo-Riemannian  Lie algebras that  satisfy the conditions in the following definition.  

\begin{definition}\label{defn: Iwasawa} 
Let $(\fraks, \ip)$ be a  solvable pseudo-Riemannian  metric Lie algebra. 
We say $(\fraks, \ip)$ is of   {\em strong Iwasawa type} if
\begin{enumerate}
\item  $\fraks = \fraka \oplus \frakn$, where  $\n = [\s, \s]$, and $\a$, is the orthogonal complement of $\n$ and is abelian, 
\item for all $A$ in $\fraka,$  $\ad_A$ is symmetric, and $\ad_A$ is zero if and only if $A=0$, 
  \item there exists an $A^0$ in $\fraka$ so that  the restriction of $\ad_{A^0}$ to $\frakn$ has  all positive eigenvalues 
  and
\item the scalar product is positive definite on $\fraka.$
\end{enumerate} 
\end{definition}
The definition of Iwasawa type (parts (i)-(iii)) is standard in the Riemannian case; we require the additional hypothesis (iv) in order for reflections in $\fraka$ to be well-defined in the pseudo-Riemannian setting. 

For any $\alpha$ in $\fraka^\ast,$  the subspace $\frakn_\alpha$ in $\n$ is defined by
\[ \n_{\alpha} = \{X \in \n \mid \ad_A(X) = \alpha(A) X ~{\rm for~all} ~
A\in\a \}.\] 
When $\n_{\alpha} $ is nontrivial, we say $\alpha$ is a {\em root} and $\frakn_\alpha$ 
is the {\em root space} for $\alpha$.  
Let $\Delta \subseteq\fraka^*$ denote the set of roots of 
$\s$.
When $(\fraks, \ip)$ is of
strong Iwasawa type, $0$ is not a root.
By properties (i) and (ii)
of strong Iwasawa type, the root space decomposition  $\n = \bigoplus \n_\alpha$ is an orthogonal direct sum 
and the scalar product is  nondegenerate on each summand.  
For each
root $\alpha$, let $H_\alpha$ in $\fraka$ denote the {\em root
vector}  defined by the property that for every $A$ in $\a$, $\la H_\alpha, A \ra = \alpha(A)$.
For  any basis $\{\alpha_i \}$ of $\fraka^\ast$, define the dual  basis  $\{ B_{\alpha_i}\}$ of $\a$  by the property that 
$\alpha_i(B_{\alpha_j}) = \delta_{ij}$ for all $i$ and $j$.

For a metric Lie algebra $(\frakg, \ip),$ and $X$ in $\frakg$, let
$\ad_X: \frakg \to \frakg$ denote the map  $\ad_X (Y) = [X,Y]$, where
 $Y \in \frakg,$ and let $\ad_X^\ast: \frakg \to \frakg$ denote its
adjoint map.  At times we will consider the adjoint map for a subalgebra
$\frakh$ endowed with the restricted inner product, and we will use
the notation $\ad_X^{\ast,\frakh}: \frakh \to \frakh$ to make  the domain of
the map clear.

  The dual map from $\fraka$ to $\fraka^*$ is defined by $ A \mapsto \check{A}$ where 
  $\check{A} (B) = \la A , B \ra$. 
The scalar product $\ip$ on $\fraka$ induces a scalar product on $\fraka^\ast,$ with $\la  A,B \ra = \la
 \check{A}, \check{B} \ra.$  For each  root
$\beta$, we have the reflection 
$\check{s}_{\beta} : \a^* \to \a^\ast$ defined by
 \[ \check{s}_{\beta} ( \alpha ) = \alpha  - \frac{2\la  \alpha,\beta
     \ra}{\la \beta, \beta  \ra}  \beta, \quad \text{for any $\alpha
     \in\fraka^\ast,$ and}   \]
the root vector $H_\beta$ induces a reflection $s_{\beta}: \fraka \to \fraka$ with
      \[ s_{\beta} (A) = A -  \frac{2 \la A, H_{\beta}\ra}{\la
          H_{\beta}, H_{\beta} \ra }  H_{\beta},\quad  \text{for any $A \in\fraka.$} 
      \]
\subsection{Geometry of  solvable metric Lie algebras}\label{sec: Ricci solv}

In this section we define geometric quantities associated to metric
Lie algebras. The definitions coincide with analogous definitions for
the associated homogeneous spaces.
\begin{defn}\label{defn: mean curvature} 
For a solvable pseudo-Riemannian metric Lie algebra $(\fraks=\fraka \oplus \frakn,\ip)$ 
of strong Iwasawa type, the {\em mean curvature vector} $H$ for $\frakn$ is
the unique vector in $\fraks$ so that for all $X$ in $\s$,
$\langle H,X\rangle =\tr\,\ad_X$. 
\end{defn}  
The Ricci curvature  for a solvable pseudo-Riemannian  Lie
algebra may be computed using the following theorems, each extended 
from the Riemannian setting to the pseudo-Riemannian setting by Conti and Rossi.  

\begin{theorem}\label{Alekseevskii}\cite{alekseevski-75,conti-rossi-19}
  Given a nilpotent pseudo-Riemannian  metric Lie algebra  $(\n, \ip)$  with an orthonormal basis $\{E_i\}$,
the Ricci  endomorphism $\Ric^{\n}$ is given by
 \begin{equation}\label{eqn: ricci-nilpotent}\Ric^{\n} = \tfrac 14 \sum \epsilon_i (\ad_{E_i}) \circ (\ad_{E_i})^* - \tfrac 12 \sum \epsilon_i (\ad_{E_i})^* \circ (\ad_{E_i}).\end{equation}
  \end{theorem}
 
The Ricci  endomorphism for a solvable  pseudo-Riemannian  Lie
algebra  may be computed using the next theorem, combined with the previous theorem.
 \begin{theorem}\label{thm: Ricci}\cite{wolter-91,
     conti-rossi-22} 
   Let $(\s=\a \oplus \n, \ip)$
   be a solvable pseudo-Riemannian  metric Lie algebra such that  
   $\ad_A$ is symmetric for all $A \in \a$ 
   with mean curvature vector $H$.  Let 
   $ \ric^{\n}$ denote the Ricci curvature of $(\n, \ip|_{\frakn})$. The Ricci curvature of $(\s , \ip)$  satisfies 
 \begin{enumerate}
 \item[(1)] $\ric (A,A') = -\tr(\ad_A \circ \ad_{A'})$ for all $A,A' \in \a$,
\item[(2)]  $\ric (X,A) =0$ for all $A \in \a$ and all $X \in\n$,
 \item[(3)]  $\ric (X,Y) = \ric^{\n} (X,Y) - \langle \ad_{H} X,Y \rangle$ for all $X,Y \in \n$.
 \end{enumerate}
 \end{theorem}
 
The connection for a metric Lie algebra $(\frakg,\ip)$ decomposes  into skew-symmetric and symmetric parts: $\nabla_XY = -\tfrac 12[X,Y]+U(X,Y),$ where $U$ is the symmetric tensor with 
\begin{equation}\label{eqn: symm} \langle U(X,Y),Z\rangle = \smallfrac{1}{2}  (\langle[Z,X],Y\rangle + \langle [Z,Y],X\rangle)\end{equation} for  $X, Y, Z \in \frakg$.  
 It follows that 
 $ U(X,X)= -{\ad_X}^\ast X$ for any $X$ in $\frakg.$
 
 Let $\frakg'$ be a subalgebra of a metric Lie algebra $(\frakg,\ip),$
 The second fundamental form $h$  
for $\frakg'$ is given by $$h(X,Y) = \nabla_X Y - \nabla'_X Y =
U(X,Y) - U'(X,Y) ,$$ for $X,Y \in \frakg',$ where 
$\nabla'$ denotes the connection of $(\frakg', \ip')$
 and  $U'$ denotes the symmetric part of 
 $\nabla'$.
The subalgebra $\frakg'$  is {\em minimal} if the trace
of $h$ is zero, and it is {\em totally geodesic} if $h$ is
identically zero.

\section{Attached subalgebras and submanifolds}\label{section: attached}

In this section, we define attached subalgebras, a notion due to
Tamaru (\cite{tamaru-08}), and we describe some of their algebraic
and geometric properties.

\subsection{The definition of attached subalgebra}\label{subsection: definitions}
Let $(\fraks=\a \oplus \n,\ip)$ be a solvable pseudo-Riemannian metric  Lie algebra of strong Iwasawa type with
root space decomposition  $\n = \sum_{\alpha \in \Delta} \n_\alpha$.  Let $\Lambda = \{\alpha_i \}$ be a finite subset  of $\Delta$ such that
(i) $\Lambda$ is a basis for $\a^\ast$, and (ii) every element in $\Delta$ can be
expressed as a linear combination of  elements of $\Lambda$ with coefficients in
$\boldZ_{\geq 0}$. (The set $\Lambda$ is analogous to a set of simple 
roots in an abstract root
system.)
Let   $\Lambda'$ be a proper subset of $\Lambda$ and enumerate its complement in
$\Lambda$ as  
\begin{equation}\label{eqn: dual basis prime} \Lambda \setminus \Lambda' = \{\alpha_{i_1}, \alpha_{i_2}, \dots, \alpha_{i_k} \} .\end{equation}
Let $\{B_{\alpha_{i}} \}$ be the dual basis to $\Lambda.$
Define $Z$ in $\fraka$ by 
\begin{equation}\label{eqn: char-vector}
Z = B_{\alpha_{i_1}} + B_{\alpha_{i_2}} + \dots + B_{\alpha_{i_k}}.\end{equation}
Note that $\alpha(Z)> 0$ for all $\alpha$ in $\Lambda \setminus \Lambda^\prime$, $\alpha(Z)=0$ if and only if $\alpha$ is a sum of elements in $\Lambda'$, and $\alpha(Z) \geq 0$ for all $\alpha$ in $\Delta$. 

Assume that (i) for all  $\alpha_j$ in $\Lambda'$, the corresponding  reflection  $\check{s}_{\alpha_j}$   permutes 
the elements of the set of roots $\alpha$ for which 
$\alpha(Z)>0$ and in addition, (ii) if $\alpha(Z)>0$, then
$\dim(\n_\alpha) = \dim(\n_{\check{s}_{\alpha_j}(\alpha)})$.
 The subset $\Lambda'$ defines a metric subalgebra $(\fraks', \ip')$ of  $(\fraks,\ip)$ as follows.
The underlying subalgebra $\fraks'$ is defined by $\fraks' = \a' \oplus \n',$ where 
 \begin{equation}\label{defn: fraka-prime-frakn-prime}
 \fraka' = \myspan_{\R}\{B_{\alpha_{i_j}}  \mid \alpha_{i_j}(Z)>0\} 
 \subseteq \fraka, \quad \text{and} \quad 
 \frakn' = \sum_{\alpha(Z) > 0} \n_\alpha \subseteq \n. 
\end{equation}
The scalar product for  $\fraks'$ is  the restriction of the scalar product for $\s$.

We can write $\fraka$ and $\frakn$ as direct sums $\fraka = \fraka' + \fraka_0$ and $\frakn = \frakn' \oplus \frakn_0$, where complementary
subspaces to $\fraka'$ in $\fraka$ and $\frakn'$ in $\frakn'$ are
defined by
 \[
 \a_0 = 
  \myspan_{\R}\{B_{\alpha_{i_j}} \mid \alpha_{i_j}(Z)=0  \}  \subseteq \a, \quad \, \text{and} \quad
  \n_0 = \sum_{\alpha(Z)=0} \n_{\alpha} \subseteq \n.
\]
We say an orthogonal basis of unit vectors $\{E_j'\} \cup \{E_j^\perp\}$ 
for $\frakn$ is a {\em compatible basis} if each basis vector lies in 
a root space;  $\{E_j'\}$ is a basis for 
$\frakn'$; and $\{E_j^\perp\}$ is a basis for $\frakn_0$.  
Such a basis exists only if the scalar product is nondegenerate on each root space. This holds in our situation because the symmetric operators $\ad_A$ are orthogonally diagonalizable.
\begin{definition}\label{defn: attached} 
For a solvable pseudo-Riemannian metric  Lie algebra of strong Iwasawa 
type $(\s, \ip)$ and $\Lambda'$ as above,  the metric Lie algebra 
$(\s', \ip')$ is  the {\em attached subalgebra defined by $\Lambda'$.}  
Let $(S,g)$ be the simply connected  solvmanifold  associated to 
$(\s, \ip)$.  The submanifold $(S',g')$ defined by $(\s', \ip')$ is 
the {\em attached submanifold defined by $\Lambda'$. } 
\end{definition}
Intuitively, an attached subalgebra is obtained by selecting a 
subset of simple roots and taking  the ideal generated by those root spaces, 
along with a compatible subspace of 
$\fraka.$
\subsection{Properties of attached subalgebras}
First, note that the restricted scalar product on $\fraks'$ is nondegenerate. 
It is nondegenerate on $\fraka'$, as $\ip$ is positive definite  on $\fraka.$ 
Because $\frakn'$ is the sum of orthogonal root spaces on which $\ip$ is nondegenerate,
$\ip$ is nondegenerate on $\frakn'.$  

An attached subalgebra of a solvable  pseudo-Riemannian metric Lie algebra of strong
Iwasawa type is also of strong Iwasawa type.
\begin{prop}\label{prop: alg properties} 
Let $(\fraks,\ip)$ be a solvable pseudo-Riemannian metric  Lie algebra of strong Iwasawa
type, and let  $(\s', \ip')$ be an attached subalgebra defined by $\Lambda'$ as in
Definition \ref{defn: attached}. 
Then
\begin{enumerate}
\item{$\n'$ is an ideal in $\s$,}
\item{$[\s',\s']= \n'$,}
\item{$[\fraka^\prime, \frakn_0] = \{0\}$, }
\item{For any $X \in \n_0$, the mapping $(\ad_X)^{\ast,\fraks}$ sends $\frakn'$ 
into $\n'$.} 
\end{enumerate}
\end{prop}

\begin{proof}
(i) Because  $\frakn'$ is
a sum of root spaces, $[\a,\n']\subseteq \n'$. To show that
$[\n,\n']\subseteq \n'$, let $X$ be in $\frakn$ and let $Y$ be in
$\frakn'$.  As $\frakn$ and $\frakn'$ are sums of root spaces, we may
assume  that $X$ and $Y$ are in root spaces
$\n_{\alpha}$ and $\n_{\beta}$ respectively, where $\alpha(Z)\geq 0$ and $\beta(Z)>0$. 
 Then $[X,Y] \in [\frakn_\alpha, \frakn_\beta]
\subseteq \frakn_{\alpha+\beta}$.
Since $(\alpha+\beta)(Z)>0$, we know 
$\frakn_{\alpha+\beta} \subseteq \frakn'$. This shows $[\n,\n']\subseteq \n'$.

(ii) Because $\frakn'$ is a
subalgebra,  $[\n',\n']\subseteq \frakn'$. In addition, $\fraka'$ is abelian and as $\frakn'$ is a sum of root spaces, $[\fraka',\frakn'] \subseteq \frakn'$.
Hence $[\s',\s'] \subseteq \frakn'$.
Now we show $\n'$ is contained in the commutator $[\fraks',\fraks']$. Because
$\n' = \oplus_{\alpha(Z)>0} \frakn_\alpha,$ it suffices to show that each root space
$\n_{\alpha}$ with $\alpha(Z) >0$ is in $[\fraks',\fraks']$.  Let
$\alpha$ be a root with $\alpha(Z)>0$ and let $X\in \n_\alpha$.  Then
$[B_\alpha,X] =\alpha(B_\alpha)X =X$, 
and
thus $X \in [\s',\s']$. 
Hence $[\fraks', \fraks'] = \frakn'$. 

(iii) Let $A$ be in $\fraka'$ and let $X$ be an
element of a root space $\n_{\beta}$ in $\n_0$. Because $\beta(Z)=0,$ it follows that $\beta(B_\alpha)=0$ for all $\alpha \in \Lambda \setminus \Lambda'.$  Write  $A =\sum_{\alpha \in \Lambda \setminus \Lambda'} c_{\alpha} B_{\alpha}$. Then 
\[ [A, X] =  \sum_{\alpha \in \Lambda \setminus \Lambda'} c_{\alpha} \, [B_{\alpha}, X ] = \sum_{\alpha \in \Lambda \setminus\Lambda'}  c_{\alpha} \, \beta(B_{\alpha})X = \sum_{\alpha_i \in \Lambda'}  (c_{\alpha} \cdot 0) \, X =0. \]

(iv) Let $X$ be an element of a root space $\n_{\beta}$ contained in $\n_0$. Let $Y$ be in a root space $\n_{\alpha}$ in $\n'$.  By definitions of $\frakn'$ and $\frakn_0$, $\alpha(Z)>0$  and $\beta(Z) = 0$.  
Because $\fraks$ is orthogonally graded by the sum $\fraks = \fraka + \sum_{\alpha \in \Delta} \n_{\alpha}$, 
we know $(\ad_X)^* Y \in \n_{\alpha - \beta}$. 
Because $(\alpha - \beta)(Z)>0$,  the  root space  $\n_{\alpha - \beta}$, and hence $(\ad_X)^* Y$, is contained in $ \n'$.
\end{proof}

\begin{prop}\label{prop: sprime}
Let $(\fraks,\ip)$ be a solvable pseudo-Riemannian metric Lie algebra of strong Iwasawa type, and let  $(\s', \ip')$ be an attached subalgebra defined by 
$\Lambda'$.  
Then  $(\s', \ip')$ is  of strong Iwasawa type.  
\end{prop}

\begin{proof} 
First  we address Property (i).
 Because $(\s, \ip)$ is of strong Iwasawa type, $\fraka$ is abelian
and orthogonal to $\frakn$. It follows that the subalgebra $\fraka'$
is abelian and orthogonal to $\frakn'$. By Proposition \ref{prop: alg properties},
$[\s',\s']= \n'$.  
 Clearly the symmetry of  $\ad_A$  is inherited so the first part of 
 Property (ii) holds. To see  that the second part of (ii) holds, assume that
 $\ad_{A}|_{\frakn'} \equiv 0$ for some $A \in \fraka'$.  
Write $A = \sum_{\alpha \in \Lambda \setminus \Lambda'} a_\alpha B_{\alpha} $. Since $0 = [A, X] =a_{\alpha}X$ for $X \in \frakn_{\alpha}$, we have  $a_\alpha =0$ for all $\alpha \in \Lambda \setminus \Lambda'$. Thus $A=0$.
Property (iii)  holds since  $Z$ is in $\a'$, and by definition,  $\frakn'$  is precisely the sum of positive eigenspaces for $\ad_Z$.
Property   (iv) is  clearly inherited from  $(\s, \ip)$. 
\end{proof}

Following \cite{tamaru-11} we can characterize $\fraka'$ in terms of the reflections in $\Lambda'.$ 
We use the notation established at the start of this section. 

\begin{lemma}\cite[Lemma 3.8]{tamaru-11}\label{lem: root perp}  Let $(\fraks, \ip)$ be a solvable pseudo-Riemannian  metric Lie algebra of strong  Iwasawa type with attached metric subalgebra $(\fraks', \ip')$ defined by the proper subset $\Lambda'$ of $\Lambda$.
Then
$$\a' = \bigcap_{\alpha_k \in \Lambda'} \Fix(s_{\alpha_k}) = \bigcap_{\alpha_k \in \Lambda'} \ker \alpha_k .$$
\end{lemma}
\begin{proof}  
Note that the dimensions of $\fraka'$ and the intersection of kernels is the same, because the roots are independent.  Thus it suffices to show that $\a'$ is contained in $\bigcap_{\alpha_j \in \Lambda'} \Fix(s_{\alpha_j})$.  
The dual basis elements $B_{\alpha_{i_j}}$ with $\alpha_{i_j} \in \Lambda \setminus \Lambda'$ span $\a'$.  
Let  $B_{\alpha_{i_j}}$  be any such basis element. 
Let  $\alpha_k$ be an arbitrary element of $\Lambda'$.   Then   $\alpha_k (B_{\alpha_{i_j}})=\delta_{k,{i_j}} =0$, so that $s_{\alpha_k} (B_{\alpha_{i_j}})=B_{\alpha_{i_j}}$ and $B_{\alpha_{i_j}} \in \Fix(s_{\alpha_k})$. 
Because   $\bigcap_{\alpha_k \in \Lambda'} \Fix(s_{\alpha_k})$ is a subspace, linearity implies that $\a'  \subseteq \bigcap_{\alpha_j \in \Lambda'} \Fix(s_{\alpha_j}).$
\end{proof}

\subsection{The Jacobi Star Condition}

A key observation about the  Jacobi Star Condition defined in Definition \ref{jacobi star} is that it is an intrinsic property of the submanifold.
\begin{lemma}\label{jacobi-star-basis-indep}
    The definition of the Jacobi Star Condition 
is independent of the choice of compatible basis 
$\{E_j^\perp\}$ for $\n_0$. 
\end{lemma}
Observe that the left-hand side of Equation \eqref{eq:jacobi-star} 
equals 
$-\Ric^\n (X) + \Ric^{\n'} (X)$, which is basis independent, while the right-hand side  equals $-\ad_H (X) + \ad_{H'} (X)$, which is also independent of choice of compatible basis.

\iffalse
\begin{remark}\label{jacobi-star-basis-indep} The definition of the Jacobi Star Condition 
is independent of the choice of compatible basis 
$\{E_j^\perp\}$ for $\n_0$. 
To see this, we observe that the left-hand side  Equation \eqref{eq:jacobi-star} above equals 
$-\Ric^\n (X) + \Ric^{\n'} (X)$, which is basis independent, while the right-hand side  equals $-\ad_H (X) + \ad_{H'} (X)$, which is also independent of choice of compatible basis.
\end{remark}
\else

\begin{remark}
In the Jacobi Star Condition in Equation \eqref{eq:jacobi-star}, the adjoint linear map $(\ad_{E_j^\perp})^{\ast,\frakn}$ in the left-hand side is with respect to the domain
$\frakn$, while the adjoint linear map $(\ad_{E_j^\perp})^{\ast, \fraks}$ in the right-hand side  is with respect to the domain $\fraks$. 
\end{remark}

\section{Geometric results}\label{section: geometric results}
\subsection{Mean curvature vector}\label{subsection-mcv}
First, we compute  root vectors and mean curvature vectors  
 for $(\fraks,\ip)$ and $(\fraks',\ip')$.
 Let $\{E_j'\} \cup \{ E_j^\perp \}$ be a compatible basis of
$\frakn$.  Let $X$ be a unit vector in a root space $\frakn_\alpha$
with $\epsilon = \la X, X \ra$, and let $A$ be in $\a$.  Then
\[ \la -\epsilon \ad_{X}^* X, A \ra = \la \epsilon X, [A, X] \ra = \epsilon\la X, \alpha(A)X\ra = \alpha(A).\]
  Hence   the root vector for the root $\alpha$ is  
 $  H_\alpha = -\epsilon (\ad_{X})^* X,$ and in 
 particular,     
$H_\alpha = -\epsilon_j (\ad_{E_j})^* E_j$ for any basis vector $E_j$ in the root space $\frakn_\alpha$.

Since $\a$ is abelian,
$$ \tr(\ad_A|_{\s}) = \sum_j\epsilon_j \la [A, E_j], E_j \ra   = \la A, -\sum_j  \epsilon_j (\ad_{E_j})^* E_j\ra .$$  
For each root $\alpha$, summing over an orthonormal basis  for 
$\frakn_\alpha$ gives $-\sum_{E_j \in \frakn_\alpha} \epsilon_j (\ad_{E_j})^* E_j  =  (\dim \,\n_\alpha) H_\alpha.$ By Definition \ref{defn: mean curvature},  the mean curvature vector $H$ for $(\fraks,\ip)$  is
\begin{equation}\label{eqn: H}
 H = -\sum_j \epsilon_j (\ad_{E_j})^\ast E_j
 = \sum_{\alpha\in\Delta} (\dim \n_\alpha) H_\alpha .
\end{equation}
Analogously, the  mean curvature vector $H'$ for $(\fraks',\ip')$ is 
\begin{equation}\label{eqn: H'}
 H' = -\sum_j \epsilon_j (\ad_{E_j^\prime})^\ast E_j^\prime
 = \sum_{\alpha(Z)>0} (\dim \n_\alpha) H_\alpha . \end{equation}
 
In the next proposition, we verify that  the mean curvature vector $H'$  
lies in $\fraka'$.  The proposition and proof are analogous to Part (1) of Proposition 
4.3 of \cite{tamaru-11}.
\begin{prop}\label{prop: mean curvature vectors}
Let $(\s,\ip)$ be a solvable pseudo-Riemannian  metric Lie algebra of strong Iwasawa type, with $\Lambda'$ a set of  roots defining an attached subalgebra   $(\s',\ip')$. 
The mean curvature vector $H'$ for $(\fraks',\ip')$ lies in $\a'$.
\end{prop}

\begin{proof}
By Lemma \ref{lem: root perp}, in order to show that $H'$ lies in $\a'$,  it suffices to show that for any  $\alpha_j \in \Lambda'$, the corresponding reflection   $s_{\alpha_j} : \a \to \a$  fixes $H'.$ 
From the definition of attached subalgebra,
the corresponding  reflection  $\check{s}_{\alpha_j} : \a^* \to \a^*$  permutes the roots for which $\alpha(Z) >0$.  

Applying $s_{\alpha_j}$ to $H'$ gives
\[ s_{\alpha_j}(H') = s_{\alpha_j} \left(\sum_{\alpha(Z)>0} (\dim\,\n_\alpha) H_\alpha \right) 
= \sum_{\alpha(Z) >0} (\dim\,\n_\alpha) s_{\alpha_j}(H_\alpha). \]
By  Definition \ref{defn: attached},  for all $\alpha$ in the sum, we have that $(\check{s}_{\alpha_j}(\alpha))(Z) >0$ and $\dim(\n_\alpha) = \dim(\n_{\check{s}_{\alpha_j}(\alpha)})$.  
It is straightforward to see that  $s_{\alpha}(H_\beta)=H_{\check{s}_\alpha(\beta)}$. Since the reflection $s_{\alpha_j}$ permutes the terms in the sum, we get 
$$s_{\alpha_j}(H') = \sum_{\alpha(Z)>0} (\dim\,\n_{\alpha}) H_{\check{s}_{\alpha_j}(\alpha)} = 
\sum_{\alpha(Z) >0} (\dim\,\n_{\check{s}_{\alpha_j}(\alpha)}) H_{\check{s}_{\alpha_j}(\alpha)} = H'.$$
This proves that  $H'$ is in $\a'$.  
\end{proof}

\subsection{The restriction of the Ricci endomorphism}

The following theorem is a generalization of Lemma 5.2 in \cite{tamaru-11}. It describes the difference in Ricci curvatures for the ambient metric Lie algebra and an attached subalgebra.
\begin{theorem}\label{thm: ricci-n-diff} 
   Let $(\fraks = \a \oplus \n,\ip)$ be a solvable  pseudo-Riemannian metric Lie algebra
   of strong Iwasawa type. 
   Let $(\fraks'= \fraka' \oplus \frakn', \ip')$ be an attached metric subalgebra defined by a set of roots $\Lambda'$
   as in Definition \ref{defn: attached}.
   Let $\{E_j'\} \cup \{E^\perp_j\}$ be a compatible orthonormal basis for $\n = \n' \oplus \n_0$.
   Then for $X$ in $\n',$  
   \[ \Ric^\n (X)  - \Ric^{\n'} (X)   = {\smallfrac{1}{2}} 
 \sum  \epsilon_j \left(  (\ad_{E_j^\perp}) \circ  (\ad_{E_j^\perp})^* -   (\ad_{E_j^\perp})^* \circ  (\ad_{E_j^\perp}) \right) (X).  
\]
 If in addition the Jacobi Star Condition holds, then
\[ \Ric^{\n} (X) - \Ric^{\n'} (X)  =  \ad_{H-H'}(X). 
\]
\end{theorem}

\begin{proof} We  compute  $\Ric^\n (X) - \Ric^{\n'} (X)$ for $X$ in $\n'.$
When we compute  $\Ric^{\n'},$ we will use that the  map $(\ad_{E_i^\prime})^{\ast, \frakn'}:\frakn' \to \frakn'$ 
 with respect to the restricted scalar product on $\frakn'$ is
equal to $\pi_{\n'} \circ (\ad_{E'_i})^{\ast,\frakn}$, where the adjoint map is relative to $(\frakn,\ip).$
By Equation \eqref{eqn: ricci-nilpotent},
\begin{align} \Ric^{\n} (X) &= \tfrac 14 \sum \epsilon_i  (\ad_{E_i}) \circ (\ad_{E_i})^* (X) - \tfrac 12 \sum \epsilon_i (\ad_{E_i})^* \circ  (\ad_{E_i})  (X), \, \text{and}
\notag \\
 \Ric^{\n'} (X) &= \tfrac 14 \sum \epsilon_i  (\ad_{E'_i}) \circ
 \pi_{\frakn^\prime} ((\ad_{E'_i})^* (X)) 
 - \tfrac 12 \sum \epsilon_i \, \pi_{\n'}((\ad_{E'_i})^* \circ  (\ad_{E'_i})  (X)). \notag 
\end{align}
We rewrite $ \Ric^{\n} (X)$ to sum over $\{ E_i^\prime\}$ and $\{E_i^\perp\}$ separately, to get
\begin{align*}  \Ric^{\n} (X) &= \tfrac 14 \left(\sum \epsilon_i  (\ad_{E_i^\perp}) \circ (\ad_{E_i^\perp})^* (X) + 
\sum \epsilon_i  (\ad_{E'_i})  \circ (\ad_{E'_i})^* (X) \right)    \\
& \qquad - \tfrac 12\left(\sum\epsilon_i  (\ad_{E_i^\perp})^* \circ  (\ad_{E_i^\perp})  (X) + 
\sum \epsilon_i  (\ad_{E'_i})^* \circ  (\ad_{E'_i})  (X)  \right). 
\end{align*}
We further expand  the second sum in each line above, using the decomposition
$(\ad_{E_i^\prime})^\ast(X) = \pi_{\frakn^\prime}
( (\ad_{E_i^\prime})^\ast (X) ) + 
\pi_{\frakn_0} ( (\ad_{E_i^\prime})^\ast(X) )$
to obtain
\begin{align}  \Ric^{\n} (X) 
&= \tfrac 14\left(A + \sum \epsilon_i  (\ad_{E'_i})  \circ \pi_{\n'}((\ad_{E'_i})^*
  (X)) +    B \right) 
\notag \\
& \qquad - \tfrac 12 \left(C + \sum  \epsilon_i \, \pi_{\n'}((\ad_{E'_i})^* \circ
  (\ad_{E'_i}) (X))  + D \right),\notag \end{align}
where $A, B, C$ and $D$ are defined by 
  \begin{alignat}{2} 
A &= \sum \epsilon_j (\ad_{E_j^\perp}) \circ  (\ad_{E_j^\perp})^* (X),& \qquad  
B &=  \sum \epsilon_i  (\ad_{E'_i}) \circ \pi_{\n_0}((\ad_{E'_i})^* (X)), \notag \\
C &=  \sum\epsilon_j  (\ad_{E_j^\perp})^* \circ  (\ad_{E_j^\perp})  (X),& \qquad
D &= \sum \epsilon_i \, \pi_{\n_0}((\ad_{E'_i})^* \circ (\ad_{E'_i}) (X)). \notag
\end{alignat}
The difference  $\Ric^{\n} (X) - \Ric^{\n'} (X) $ may  be written as 
\[ \Ric^{\n} (X) - \Ric^{\n'} (X)  = \smallfrac 14 A + \smallfrac 14  B - \smallfrac 12 C -\smallfrac 12 D. \]
We now simplify the terms $A, B, C$ and  $D$.  
First observe  
\begin{equation}\label{little identity} \langle (\ad_{E'_i})^*X, E_j^\perp \rangle = \langle X, [E'_i,
E_j^\perp] \rangle = -\langle X, [E_j^\perp , E'_i] \rangle = -\langle
(\ad_{E_j^\perp})^* X, E'_i \rangle.\end{equation}  
To begin, we show $B = A$: 
\begin{alignat*}{2}
 B &= \sum_i \epsilon_i  (\ad_{E'_i}) \circ \pi_{\n_0} ((\ad_{E'_i})^* X) \qquad &
  \text{by definition}\\ 
& = 
\sum_i\epsilon_i  [ E'_i, \sum_j \epsilon_j \langle (\ad_{E'_i})^* X, E_j^\perp \rangle
E_j^\perp] \qquad &  \strut \text{using the basis for $\frakn_0$}  \\
&=  \sum_i \epsilon_i  \sum_j \epsilon_j [ \langle (\ad_{E_j^\perp})^* X, E'_i \rangle
E_j^\perp, E'_i ] 
& \text{from Equation \eqref{little identity}}\\
&= \sum_j \epsilon_j [ E_j^\perp , \sum_i \epsilon_i \langle (\ad_{E_j^\perp})^* X,
E'_i\rangle E'_i]  &  \\
& = \sum_j \epsilon_j [E_j^\perp, (\ad_{E_j^\perp})^* X] & 
\text{because $X \in \frakn'$}\\ 
&= \sum_j \epsilon_j  (\ad_{E_j^\perp}) \circ  (\ad_{E_j^\perp})^* (X) 
= A. & \end{alignat*}

Next we prove $D=0$ by proving that each term in the sum vanishes.
We may assume $X \in \n_{\alpha}$ with $\alpha(Z) > 0$.  Take any $E'_i \in \n_{\beta}$ with $\beta(Z) > 0$. Then  $(\ad_{E'_i}) X \in \n_{\alpha + \beta}$. Hence $(\ad_{E'_i})^* (\ad_{E'_i}) X \in
\n_{\alpha} \subseteq \n'$, which implies its projection onto $\n_0$ vanishes: $\pi_{\n_0}((\ad_{E'_i})^* (\ad_{E'_i}) X) =0$. This shows that $D=0$.

Having shown that $A=B$ and $D=0$,  we see 
\[ \Ric^{\n} (X) - \Ric^{\n'} (X)  = \tfrac 14 A + \tfrac 14  B - \tfrac 12 C -\tfrac 12 D  = \tfrac 12 (A -  C).
\]
Substituting for $A$ and $C$ yields  the desired expression for $\Ric^{\n} (X) - \Ric^{\n'} (X) $.

We note that $\tfrac 12 (A -  C)$ is precisely the left-hand side of the Jacobi Star Condition in Equation \eqref{eq:jacobi-star}, while $\ad_{H-H'}(X)$ is precisely the right-hand side of the Jacobi Star Condition. Thus when the Jacobi Star Condition holds, we have
\[ \Ric^{\n} (X) - \Ric^{\n'} (X)  =  \ad_{H-H'}(X). 
\]
\end{proof}

The next corollary is the counterpart 
to Theorem 5.3 in \cite{tamaru-11}. Any attached subalgebra of an Einstein metric Lie algebra meeting the hypotheses of Theorem \ref{thm: ricci-n-diff}  will also be Einstein.

\begin{coro}\label{cor: Einstein}
   Let $(\fraks,\ip)$ be a solvable pseudo-Riemannian
   metric  Lie algebra of strong Iwasawa type. 
    Let $(\fraks', \ip')$ be the attached metric subalgebra defined by $\Lambda'$ 
   as in Definition \ref{defn: attached}. 
 If the Jacobi Star Condition holds, and 
   if $(\fraks,\ip)$ is Einstein with Einstein constant $\lambda$,  then $(\fraks',\ip')$ is Einstein with Einstein constant $\lambda$. 
\end{coro}

In the proof of this corollary, we will use a characterization of Einstein scalar products from \cite{conti-rossi-22}, presented in Theorem \ref{thm: 3.9 conti-rossi} below.   The hypotheses of this theorem require that the solvable 
  metric Lie algebra in question admit 
  a {\em pseudo-Iwasawa decomposition}:  $\fraks$ is the sum of an abelian subspace $\fraka$ and a nilpotent ideal; and  $\ad_A = (\ad_A)^*$ for all $A \in \a$. 
By properties (i) and (ii) in Definition \ref{defn: Iwasawa}, for any  solvable pseudo-Riemannian  metric Lie algebra of strong Iwasawa type, the decomposition $\fraks = \fraka \oplus [\fraks, \fraks]$ is a pseudo-Iwasawa decomposition.

The trace form for a metric Lie algebra $(\frakg,\ip)$  is defined by $\la F,G\ra_{\trace}^{\End(\frakg)} = \trace (F \circ G)$
  for $F$ and $G$ in $\End(\frakg)$. The identity map for $\frakg$ is denoted by $\Id^{\frakg}$.

\begin{theorem}\cite[Theorem 3.9]{conti-rossi-22}\label{thm: 3.9 conti-rossi}
Let  $(\fraks,\ip)$ be a solvable metric Lie algebra with a pseudo-Iwasawa decomposition. 
Then 
  $(\fraks,\ip)$ is Einstein with Einstein constant $\lambda$ if and only if $\Ric^{\frakn} = \lambda \Id^{\frakn} + \ad_H$  and
  $\la \ad_A, \ad_B \ra_{\trace}^{\End(\frakn)} = \lambda \la A, B \ra$ for all $A$, $B$ in $\fraka.$
\end{theorem}

Now we are ready for the proof of the corollary. 
\begin{proof}[Proof of Corollary \ref{cor: Einstein}] Suppose that $(\fraks,\ip)$ and  $(\fraks', \ip')$ satisfy the hypotheses of Corollary \ref{cor: Einstein}.
Since $(\fraks,\ip)$ is Einstein with Einstein constant $\lambda$,  by Theorem \ref{thm: 3.9 conti-rossi},  $\Ric^{\frakn} = \lambda \Id^{\frakn} + \ad_H$  and
  $\la \ad_A, \ad_B \ra_{\trace}^{\End(\frakn)} = \lambda \la A, B \ra$ for all $A$ and $B$ in $\fraka.$

By Theorem \ref{thm: ricci-n-diff}, $\Ric^{\frakn'}= 
 \Ric^\frakn - \ad_{(H-H')}.$ 
Restricting to $\frakn'$ gives the equality  
$\Ric^{\frakn'} - \ad_{H'} =\Ric^\frakn - \ad_H = \lambda\Id^{\frakn}$
and thus 
$\Ric^{\frakn'} =  \lambda\Id^{\frakn'} + \ad_{H'}$.
  
  Choose a compatible orthonormal basis $\{E_i\} = \{E_i^\prime\} \cup \{E_i^\perp\}  $ for $\frakn = \frakn' \oplus \frakn_0.$
  By Proposition \ref{prop: alg properties}, $[\fraka', \frakn_0] = \{0\} $, so for $A$ and $B$ in $\fraka'$, 
  $\ad_A (\ad_B E_i^\perp) = 0$ for all $i.$
Using this, we rewrite the trace form
  \begin{align*}
   \trace(\ad_A|_{\frakn} \circ \ad_B|_{\frakn})
   &= \sum \epsilon_i \la \ad_A  \ad_B E_i', E_i' \ra  + \sum \epsilon_i \la \ad_A  \ad_B E_i^\perp, 
   E_i^\perp \ra \\
   &=  \sum \epsilon_i \la \ad_A  \ad_B E_i', E_i' \ra \\
   &=
  \trace(\ad_A|_{\frakn'} \circ \ad_B|_{\frakn'}). 
\end{align*}
 Thus, 
 $ \la \ad_A, \ad_B \ra_{\trace}^{\End(\frakn)} =  \la \ad_A, \ad_B \ra_{\trace}^{\End(\frakn')}.$
Because the scalar product on $\fraks'$ is the restriction of the scalar product on $\fraks$, $\la A, B\ra =
\la A, B\ra'$ for any $A$ and $B$ in $\fraka'$.  Hence,  $ \la \ad_A, \ad_B \ra_{\trace}^{\End(\frakn')} = \lambda\la A,B \ra' $.

By  Theorem \ref{thm: 3.9 conti-rossi},  $(\fraks', \ip')$ is Einstein with Einstein constant  $\lambda$.
\end{proof}

We now prove the main theorem.
 \begin{proof}[Proof of Main Theorem:] We first show that (i) and (ii) are equivalent using  Theorem \ref{thm: Ricci}.  Let $A$ and $B$ be in $\fraka'$ and let  $X$ and $Y$ be in $\frakn'$. First observe that 
 \[  \ric^\s (A,X) = 0 = \ric^{\s'} (A,X).\]
 Next, we have
\begin{align*}
   \ric^\s (A,B) - \ric^{\s'} (A,B) &= \
   \trace^{\frakn} (\ad_A \circ \ad_B) -  \trace^{\frakn'
   } (\ad_A \circ \ad_B) 
   \\ &= \
   \trace^{\frakn_0} (\ad_A \circ \ad_B). 
\end{align*} 
When  $X \in \frakn_0,$  by Part (iii) of Prop \ref{prop: alg properties}, $[A, X] =0$ so $\trace^{\frakn_0} (\ad_A \circ \ad_B) =0$ and $\ric^\s (A,B) = \ric^{\s'} (A,B).$
Last,
  \[  \ric^\s (X,Y) - \ric^{\s'} (X,Y) = \ric^\n (X,Y) - \ric^{\n'} (X,Y) -\left(\langle \ad_H X, Y\rangle - \langle \ad_{H'} X,Y \rangle\right).\]
Hence the Ricci curvatures $\ric^\s$ and $\ric^{\s'}$
agree if and only if 
\[  \ric^\n (X,Y) - \ric^{\n'} (X,Y) =  \langle \ad_{H - H'} X, Y \rangle.\]

We show (ii) and (iii) are equivalent.  Let $\{E_j\} = \{ E_j^\perp \} \cup \{E_j'\}$ denote a compatible orthonormal basis of $\frakn = \n' \oplus \n_0$. 
Using Equations \ref{eqn: H} and \ref{eqn: H'},
\[  H - H' = \sum_{\alpha(Z)=0} (\dim \n_\alpha) H_\alpha =  -\sum  \epsilon_j (\ad_{E_j^\perp})^{\ast, \fraks} E_j^\perp. \]
 Combining this with Theorem \ref{thm: ricci-n-diff} we rewrite  the Jacobi Star Condition as 
\begin{align*} 
     \ad_{H -H'} X &=  \tfrac12 \sum \epsilon_j \left( (\ad_{E_j^\perp}) \circ  (\ad_{E_j^\perp})^* -   (\ad_{E_j^\perp})^* \circ  (\ad_{E_j^\perp})\right) X \\
 &=  \Ric^\n (X)  - \Ric^{\n'} (X) 
\end{align*}
for all $X$ in $\frakn'$.
Thus $\Ric^{\n} (X) - \Ric^{\n'} (X) =  \ad_{H - H'}(X)$ 
for all $X$ in $\frakn'$ if and only if the Jacobi Star Condition holds.  
\end{proof}

\subsection{Subalgebra geometry}

In this section, we show that an attached  submanifold  of an Einstein solvmanifold is a minimal submanifold. We also prove that 
the attached submanifold 
 is  totally geodesic  if and only if $\Lambda'$ and $\Lambda \setminus \Lambda'$ are orthogonal in $\fraka^*$.  Recall that $U(X,Y)$ is the symmetric part of $\nabla_X Y$  (cf. Equation (\ref{eqn: symm})).

\begin{lemma}\label{lem: ad-starX-X in n}
Let $(\s=\a \op \n,\ip)$ be a solvable pseudo-Riemannian metric  
Lie algebra  of strong Iwasawa type and let $A$ be in $\fraka$ and let 
$X$ and $Y$ be in root spaces $\frakn_\alpha$ and $\frakn_\beta$ respectively.
\begin{enumerate}
    \item $U(A,A)=0$,
    \item $U(A,X)=-\smallfrac{1}{2} \ad_A X$ lies  in $\frakn_\alpha$, 
    \item $U(X,X)= (\ad_X)^{\ast,\fraks}X$ lies in $\fraka$, and
    \item When $\alpha \neq \beta$, $\pi_\fraka(U(X,Y))=0$ and $\pi_\frakn ( U(X,Y) ) \in \frakn_{\beta - \alpha} + \frakn_{\alpha - \beta}$.
\end{enumerate} 
 \end{lemma}
 \begin{proof}  Let $A \in \a$, let $\frakn_\alpha$ and $\frakn_\beta$ be nontrivial root spaces, and  let $X$ be in
   $\frakn_\alpha$ and let $Y$ be in $\frakn_\beta.$   
   (i)   By the orthogonality of $\a$ and $[\s,\s]$,  $U(A,A)=0$.
(ii) 
Since $\ad_A$ is symmetric, $U(A,X)=-\smallfrac{1}{2} \ad_A X$,  which is in the root
space $\frakn_\alpha$.
 (iii)   We show that $ (\ad_X)^{*,\fraks}X$ is orthogonal to $Y$.   Because 
 $(\s,\ip)$ is of strong Iwasawa type, $\beta$ is nonzero, so $\alpha \ne \alpha + \beta$.  
  Because the distinct root spaces $\frakn_{\alpha}$ and $\frakn_{\alpha+\beta}$ 
  are orthogonal, $[X,Y]$ is in $\frakn_{\alpha+\beta}$. Thus,
  \[ \la  U(X,X),Y \ra = \langle-(\ad_X)^{*,\fraks} X, Y\rangle = 
  \langle X,[X,Y]\rangle =0.\] Hence  $(\ad_X)^{*,\fraks}X$ is in $\fraka$.
  (iv) To see that $\pi_\fraka(U(X,Y))=0$,
\begin{align*}\la U(X,Y), A \ra &=\tfrac12 ( \la
[A,X],Y\ra +\la[A ,Y],X\ra) \\
&=\tfrac12(\alpha(A)\la X,Y\ra +\beta(A)\la Y,X\ra)
\\
&=\tfrac12(\alpha(A) \cdot 0  +\beta(A) \cdot 0 ) = 0.
\end{align*}
The second observation follows from the definition of $U$ and from the fact that $[\frakn_{\alpha}, \frakn_{\beta}] \subseteq\frakn_{\alpha + \beta}$.
\end{proof}

Lemma 6.1 of \cite{tamaru-11} generalizes as follows.
\begin{lemma}\label{lem: secondfundform}
Let $h$ be the second fundamental form of an attached subalgebra $(\s',\ip')$ of the metric Lie algebra $(\s,\ip)$. Then 
\begin{enumerate}
\item $h(A,A)=0$ for all $A$ in $\a'$,  \label{h on A is zero}
\item $h(A,X) =0$ for all $A$ in $\a'$ and $X$ in  $\n'$, and 
\item for each root space $\n_\alpha$  in  $\n'$,  and for all  $X$ in  $\n_\alpha$, $h(X, X) = -\pi_{\a \ominus  \a'}((\ad_{X})^{*,\fraks} X)$.\label{h root vector}
\end{enumerate}
\end{lemma}
Recall from Section \ref{sec: Ricci solv} that 
 the second fundamental form 
for $(\fraks', \ip')$ in $(\fraks,\ip)$ is given by 
$h(X,Y) =
U(X,Y) - U'(X,Y),$
for $X,Y \in \fraks',$ where $U$ is as in Equation \eqref{eqn: symm}. 

\begin{proof} By Lemma \ref{prop: sprime}, the attached subalgebra is of strong Iwasawa type.
(i) Let $A \in \a'$.  By Lemma \ref{lem: ad-starX-X in n}, since $U(A,A)=0$ and
 $U'(A,A)=0$, we have $h(A,A)=0$. 
(ii) Suppose $A \in \fraka'$ and $X$ in $\frakn'$.  By Lemma \ref{lem: ad-starX-X in n},  $U(A,X) = -\tfrac12 \ad_A X = U'(A,X)$ is in $\frakn'$.  Hence $h(A,X)=U(A,X)-U'(A,X)=0.$ 
(iii) Let  $X$ be  in  a root space of $\n'$. 
  By Lemma \ref{lem: ad-starX-X in n},  $U'(X, X) =  (\ad_{X})^{\ast, \fraks'} X$ is in  $\fraka'.$ 
Note $(\ad_{X})^{\ast, \fraks'} = \pi_{\fraks'} \circ \ad_{X}^{\ast, \fraks}$.  Hence $(\ad_{X})^{\ast, \fraks'} X =  \pi_{\a'}((\ad_{X})^{\ast, \fraks} X)$.   Thus we have
\begin{align*}
h(X,X) &=  U(X,X) - U'(X,X) \\ 
&= -(\ad_{X})^{*,\fraks} X + (\ad_{X})^{*,\fraks'} X\\
&= -((\ad_{X})^{*,\fraks} X - \pi_{\fraka'}((\ad_{X})^{*,\fraks'} X)).
      \end{align*}
Because  $(\ad_{X})^{*,\fraks} X \in \fraka$, the right side may be rewritten as $-\pi_{\a \ominus \a'}((\ad_{X})^{\ast,\fraks} X)$.
\end{proof}

The next theorem is analogous to Theorem 6.2 of \cite{tamaru-11}.
\begin{theorem}\label{thm: minimal}
An attached submanifold  of a solvmanifold  is  minimal.
\end{theorem}
\begin{proof} Let $(\fraks,\ip)$ be a solvable pseudo-Riemannian
   metric  Lie algebra of strong Iwasawa type. 
    Let $(\fraks', \ip')$ be the attached metric subalgebra.  
    Let  $\{ A_i \}$ and $\{E_j'\}$ be orthonormal bases
for $\a'$ and $\n'$, respectively.  Then
$$\trace(h) = \sum_i \epsilon_i h(A_i,A_i) + \sum_j \epsilon_j h(E_j', E_j').$$
By Lemma \ref{lem: secondfundform}, $h(A_i,A_i)=0$ for each $i$, and 
$h(E_j', E_j') = -\pi_{\a \ominus \a'}((\ad_{E_j'})^* E_j')$. Therefore,
$$\trace(h) =  \sum_j \epsilon_j  h(E_j', E_j') = -\sum_j \epsilon_j \pi_{\a \ominus \a'}((\ad_{E_j'})^* E_j').$$
By Proposition \ref{prop: mean curvature vectors}, 
the  mean curvature vector  
$-\sum_j \epsilon_j ((\ad_{E_j'})^* E_j')$ is in $\a'$.  Therefore, 
the projection $\pi_{\a \ominus \a'} \left(-\sum_j \epsilon_j (\ad_{E_j'})^* E_j' \right)  $ is zero, 
and the trace of $h$ vanishes, as claimed.
\end{proof}
 
 The next proposition, a generalization of 
  Proposition 6.4 of \cite{tamaru-11},  characterizes  totally geodesic attached submanifolds. 
\begin{prop}\label{prop: tot-geodesic}  
The attached metric subalgebra  $(\fraks',\ip)$ is totally geodesic in  $(\fraks, \ip)$ 
if and only if
  $\Lambda'$ and $\Lambda \setminus \Lambda'$ are orthogonal.
\end{prop}
\begin{proof} Suppose that  the sets $\Lambda'$ and $\Lambda \setminus \Lambda'$ are orthogonal.
By definition of the dual scalar product on $\fraka^*$, the scalar  product 
 $\la H_\alpha, H_\beta \ra =\la \alpha,\beta \ra = 0$  for any $\alpha \in \Lambda'$ and any $\beta \in \Lambda \setminus \Lambda'$. 
 By the definition of attached, $\Lambda$ is a basis for $\fraka^\ast$ and therefore,  $\{H_{\alpha_i} \mid \alpha_i \in \Lambda\}$  is a basis for $\fraka$. It follows that $\fraka$ is the orthogonal direct sum of 
 $\myspan\{H_{\alpha_j} \mid \alpha_j \in \Lambda' \}$ and $\myspan \{ H_{\alpha_i} \mid \alpha_i \in  \Lambda \setminus \Lambda'\}.$

To prove that $\s'$ is totally geodesic in $\fraks$, we show that the second fundamental form $h$ for $\s'$ vanishes. We already know from Lemma \ref{lem: secondfundform} that 
$h(A,A)=0$ for any $A \in \fraka'$ and 
$h(A,X)=0$ for any $A \in \fraka'$ and $X$ in $\frakn'$.  It remains to show that $h(X, Y) = 0$ for $X$ and $Y$ in $\frakn'$. 
By the definitions of $U$ and $U'$, the  vector $U'(X,Y)$ is the orthogonal projection of $U(X,Y)$ to $\fraks'.$  It follows from the  definition of the second fundamental form  that  to show $h(X, Y) = 0$ for $X$ and $Y$ in $\frakn'$, it suffices to show that  $U(X,Y) \in \fraks'$. 
To show  $U(X,Y) \in \fraks'$ it is enough to demonstrate that this holds for 
$X$ and $Y$ in root spaces in $\frakn'.$  

Let $X$ be in $\frakn_{\alpha_i}$ and let $Y$ be in $\frakn_{\alpha_j},$  where  $\frakn_{\alpha_i}$ and $\frakn_{\alpha_j}$ are in $  \frakn'$ (and $\alpha_i, \alpha_j  \in \Lambda \setminus
\Lambda'$).  
First we consider $h(X,Y)$ in the case that $\alpha_i = \alpha_j.$  By a 
polarization argument, it suffices to show that  $h(X,X) =0 $.  
By Part (iii) of
Lemma  \ref{lem: secondfundform}, 
 it suffices to show that $(\ad_X)^{\ast,\fraks} X \in
\fraka'$. Because  $(\ad_X)^{\ast,\fraks} X \in \fraka$,
 we want to establish that
\[ \la (\ad_X)^{\ast,\fraks} X, B_{\alpha_j} \ra = 0 \]
for $B_{\alpha_j} \in \fraka_0$ (so $\alpha_j \in \Lambda'$,
$\alpha_i \ne \alpha_j$).  But 
\begin{align*} \la (\ad_X)^{\ast,\fraks} X, B_{\alpha_j} \ra &= 
- \la  X, [B_{\alpha_j},X] \ra  \\ 
&= - \la  X, \alpha_i(B_{\alpha_j})X \ra \\
&= - \la  X, \delta_{\alpha_i,\alpha_j} X \ra  \\ 
&= 0. \end{align*}

Now suppose $\alpha_i \ne \alpha_j.$    The root spaces 
$\frakn_{\alpha_i}$ and $\frakn_{\alpha_j}$ are orthogonal.
By Lemma  \ref{lem: ad-starX-X in n}, because $X$ and $Y$ 
are in distinct root spaces, $U(X,Y)$ is orthogonal to $\fraka$. 
Now suppose that $W \in \n_\beta$ is in $\n_0$ and that
$\la U(X,Y), W \ra \neq 0.$  By Lemma \ref{lem: ad-starX-X in n}, this
is nontrivial only if $\beta = \pm(\alpha_1 - \alpha_2 )$. 
Without loss of generality, suppose that $\beta=\alpha_1 - \alpha_2.$  
In that case, the nontrivial vector 
$H_\beta = H_{\alpha_i} - H_{\alpha_j}$ is in $\fraka'$, contradicting 
the orthogonality of $\myspan\{H_{\alpha_j} \mid \alpha_j \in \Lambda' \}$
and $\myspan \{ H_{\alpha_i} \mid \alpha_i \in  
\Lambda \setminus \Lambda'\}$ established in the first paragraph. 
Thus $U(X,Y)$ is orthogonal to $\n_0$. 
We have shown that $U(X,Y)$ is orthogonal to $\fraka \oplus \frakn_0;$ 
hence it is in $\fraks'$, and $h(X,Y)=0$ as desired. 

Conversely, when $\Lambda'$ and $\Lambda \setminus \Lambda'$ are not orthogonal, there are roots 
$\alpha \in \Lambda'$ and $\beta \in \Lambda \setminus \Lambda'$ 
so that the corresponding root vectors $H_\alpha$ and $H_\beta$ are
not orthogonal. Let $E'_i \in \n_\alpha$. By Lemma \ref{lem: secondfundform}, 
$h(E'_i,E'_i) = \pi_{\a \ominus \a'}(-(\ad_{E'_i})^* E'_i) = \pi_{\a \ominus \a'}(\epsilon_i H_\alpha)$. We know that
 $\pi_{\a \ominus \a'}(\epsilon_i H_\alpha)$ is nonzero because $\la H_\alpha, H_\beta \ra$ is nonzero,   and $H_\beta$ is in $\a \ominus \a'$. Thus,  $\s'$ is not totally geodesic in 
 $\s$.
\end{proof}

\section{Special Case: Symmetric spaces}\label{section: symmetric spaces}

The Main Theorem  applies to symmetric spaces of noncompact type, the setting of \cite{tamaru-11}.
We derive Tamaru's Theorem 5.3 from  \cite{tamaru-11} as a corollary to our Main Theorem. 

\subsection{Symmetric spaces of noncompact type as solvable metric Lie algebras}\label{subsection: symmetric space}

Any symmetric space $(M,g)$ of noncompact type is isometric to a solvmanifold defined by a solvable metric Lie algebra.  Let  $M = G/K$ be a symmetric space of noncompact type, where $G$ is the connected component of the identity in the isometry group of $M.$   The corresponding Lie algebra  $\g$ is semisimple. Let $\sigma$ denote the Cartan involution of $\frakg$.  The Killing form  $B$ defines  a positive definite,  symmetric bilinear form $B_\sigma$ on $\frakg$: 
$$B_\sigma(X,Y) =  -B(X, \sigma(Y)), \text{ for any $X$ and $Y$ in $\frakg$}.$$ 
For any $X, Y$ and $Z$ in $\g$,
$B_{\sigma} ([Z,X],Y) =-B_{\sigma}(X,[\sigma(Z),Y]).$
Let $\g = \frakk \oplus \a \oplus \n$ be the Iwasawa decomposition of $\frakg.$ 
The rule 
\begin{equation}\label{eqn: symmetric ip} \la ~,~\ra = 2B_\sigma |_{\fraka \times \fraka}(~,~) +  B_\sigma |_{\frakn \times \frakn}(~,~).\end{equation} defines an inner product on the solvable subalgebra $\fraks= \a \oplus \n$. 
The  metric Lie algebra $(\fraks,\la ~,~\ra )$  
corresponds to a  solvmanifold  isometric to $(M,g)$.  

Now let  $X$ be an element of $\frakn$ and consider  $\ad_X|_{\fraks}$.   For $Y \in \frakn$ and $Z $ in $\fraks,$ 
\[ \la (\ad_X)^{*,\fraks}(Y),Z \ra = B_{\sigma}([X,Z],Y) =- B_{\sigma}(Z,[\sigma(X),Y]). \]
\begin{align*}\intertext{In particular, when $Z = A \in \fraka$, }\la (\ad_X)^{\ast,\fraks} Y,A\ra &= 
-\tfrac12\la A,[\sigma(X),Y]\ra  \notag \\
\intertext{and when $Z=V$ in $\frakn$, } \la (\ad_X)^{\ast,\fraks} Y,V\ra &= 
-\la V,[\sigma(X),Y]\ra. \notag 
\end{align*}
As a consequence,  for any $X$ and $Y$ in $\frakn$, 
\begin{align}\label{eqn: ad-star-a-n}
\ad_X^{*,\fraks}(Y) &= -\tfrac12 [\sigma(X), Y]_{\a} -[\sigma(X),Y]_{\n} \\
\ad_Y^{*,\fraks}(Y) &= -\tfrac12 [\sigma(Y), Y] \\
\ad_X^{*,\frakn}(Y) &=  -[\sigma(X),Y]_{\n} 
\end{align}
We have used the fact that $[\sigma(Y),Y]$ is in $\a$ for the second equality. 
The subalgebra $\n = \sum_{\lambda \in \Delta^+} \g_\lambda$ is the sum of the positive root spaces for the adjoint action of $\fraka$ on $\frakn.$ Let $\Lambda \subseteq \Delta^+$ be a set of simple roots. 
 It is straightforward  to verify that $(\fraks, \ip)$ is of strong Iwasawa type, and  $\Lambda$ has the properties in Definition \ref{defn: attached}.  
Hence, any proper
 subset $\Lambda'$ of $\Lambda$ defines an 
 attached solvable metric Lie algebra $(\fraks', \ip')$ as in Definition \ref{defn: attached}.

\subsection{Recovering Tamaru's Theorem 5.3}\label{subsection:  reobtaining Tamaru}
  We derive Tamaru's Theorem 5.3 as a special case of our Main Theorem.
 \begin{coro}\cite[Thm 5.3]{tamaru-11}\label{cor: hiroshi}
Let $(M,g)$ be a  symmetric space of noncompact type and let 
 the solvable  metric Lie algebra $(\fraks, \ip)$ be as described in the previous subsection so that it defines a solvmanifold isometric to $(M,g)$. Let $(\fraks',\ip')$ be an attached subalgebra.
 For any $X \in \s'$, the restriction of the Ricci endomorphism for $(\fraks,\ip)$ to $\fraks'$ coincides with the Ricci endomorphism for $(\fraks',\ip')$.  Because $(\fraks,\ip)$ is Einstein, $(\fraks',\ip')$ is Einstein with the same Einstein constant.
  \end{coro}
  \begin{proof} 
We show that the Jacobi Star Condition holds.   Let $X$ and $Y$ be in $\frakn$.  Then Equation \eqref{eqn: ad-star-a-n}  gives
\begin{alignat*}{2} 
\tfrac12 ((\ad_Y)^{\ast,\frakn} \ad_Y - \ad_Y (\ad_Y)^{\ast,\frakn})(X) &= 
 - \tfrac12 ( \ad_{\sigma(Y)}  \circ \ad_Y - \ad_Y \circ  \ad_{\sigma(Y)}) (X) \\ &= -\tfrac12 \bigl( [\sigma(Y), [Y,X]] - [Y, [\sigma(Y), X]]\bigr).
 \end{alignat*}
 Rewriting the right side
 using the Jacobi Identity and also using Equation 
 \eqref{eqn: ad-star-a-n},  we get  
 \[  \tfrac12 ((\ad_Y)^{\ast,\frakn} \ad_Y - \ad_Y (\ad_Y)^{\ast,\frakn})(X)=-\tfrac12
 [[\sigma(Y),Y], X] = 
\ad_{((\ad_Y)^{\ast,\fraks} Y)} (X). \] 
Summing 
 over appropriate basis vectors in the place of $Y$ yields
 the Jacobi Star Condition. By the Main Theorem, $\Ric^{\s}(X) = \Ric^{\s'}(X)$   for any $X \in \s'$.
  \end{proof}
  
\section{An example}\label{section: example}

To show that our generalization extends beyond symmetric spaces,
 we present an example of a solvable metric Lie algebra of strong Iwasawa type that meets the hypotheses of the Main Theorem. This example is obtained from the affine untwisted Kac-Moody
algebra $\frakg$ defined by the simple classical Lie algebra $\fraksl_3(\boldR).$ The
infinite-dimensional Lie algebra $\frakg$ has a triangular decomposition
$\frakg = \frakn_- + \frakh + \frakn_+.$ We have taken the upper
triangular part $\frakn_+$ and truncated it to obtain a
finite-dimensional nilpotent Lie algebra.  We extend by a
codimension-one subalgebra of the Cartan subalgebra $\frakh$ of $\frakg$ to define a finite-dimensional  solvable Lie algebra $\fraks$.  We endow it with a natural inner product. 

We will  first present a standalone definition of the  solvable metric Lie
algebra $(\fraks,\ip)$ obtained from $\frakg.$   Having defined
$(\fraks,\ip)$, we will then 
define the attached subalgebra $\fraks'$ defined by  a
set of roots $\Lambda'$.   We will next 
compute Ricci curvatures and verify that
the Jacobi Star Condition holds.  Last, we will show that $(\fraks,\ip)$  is
Einstein, but not a symmetric space, and that the attached metric subalgebra
is not totally geodesic. 
\subsection{Definitions}
In what follows, $E_{ij}$ will denote the matrix in $\fraksl_3(\boldR)$ with a one in the $(i,j)$ entry and zeros elsewhere, and  $H_{ij}$ will denote the diagonal 
matrix $E_{ii} - E_{jj}$.
\begin{defn}
Let $\frakn = \myspan \{ 1 \otimes E_{12}, 1 \otimes E_{13}, 1 \otimes E_{23}\} \bigoplus t \otimes \fraksl_3(\boldR)$, with Lie 
brackets be defined by 
\begin{equation}\label{eqn: n bracket} [t^i \otimes X, t^j \otimes Y]^{\frakn} =  t^{i+j} \otimes  [X,Y],\end{equation}
where $[X,Y]$ denotes the commutator of matrices $X$ and $Y$. 
Derivations $D_0, D_1, D_2$ of $\frakn$ are defined as follows: 
  \begin{align*} D_0(t^i \otimes X) &=  i (t^i \otimes  X ), \\
  D_1(t^i \otimes X) &= t^i \otimes [H_{12}, X ],\, \text{and}  \\
   D_2(t^i \otimes X) &= t^i \otimes [H_{23}, X ], \end{align*}
  for $X \in \fraksl_3(\boldR)$ and $t \in \{0,1\}$.
These derivations define a solvable extension $\fraks = \frakh + \frakn$ of $\frakn$, where  $\frakh  = \myspan \{  \bfD,~\bfH_1,~\bfH_2 \} \cong \boldR^3 $
is the vector space with basis $\{\bfD,~\bfH_1,~\bfH_2\}$.
  The Lie bracket in $\fraks$ is determined by 
\begin{itemize}
\item $[\bfD,X] = D_0(X),~ [\bfH_1, X] = D_1(X),~$ and $[\bfH_2, X]  =D_2(X)$ if $X$ is in $\frakn$, 
\item  $[X,Y] = [X,Y]^\frakn$ if $X , Y \in \frakn$, 
\item $[X,Y]=0$ if $X$ and $Y$ are in $\frakh$. 
\end{itemize}
If we write  $\bfH_1 = 1 \otimes H_{12}$ and $\bfH_2 = 1\otimes H_{23}$, then Equation \eqref{eqn: n bracket} extends to $\bfH_1$ and $\bfH_2.$
(We use bold font for elements of $\frakh$ to distinguish them from elements of $\fraksl_3(\boldR).$)

Define a scalar product on $\fraks$ by 
making $\frakh$ and $\frakn$ orthogonal;  defining the scalar  product on $\frakn$ by 
\begin{equation}\label{eqn: ad s} \la t^i \otimes X, t^j \otimes Y \ra = \delta_{ij}\tr( X^TY)\quad \text{for $X, Y \in \fraksl_3(\boldR)$ and $t \in \{0,1\}$};\end{equation}
and defining the scalar  product on $\frakh$ by
   letting  
\begin{equation}\label{eqn: mtx def}\la a_0 \bfD + a_1 \bfH_1 + a_2 \bfH_2, b_0 \bfD + b_1 \bfH_1 + b_2 \bfH_2 \ra = \begin{bmatrix} a_0 & a_1 & a_2\end{bmatrix}
\begin{bmatrix} \smallfrac{16}{9} & 0 & 0 \\ 0 & 4 & -2 \\ 0 & -2 & 4 \end{bmatrix}
\begin{bmatrix} b_0 \\ b_1 \\ b_2\end{bmatrix}.
\end{equation}
\end{defn}
It is straightforward to confirm that $(\fraks,\ip)$ is a solvable metric Lie algebra  of strong Iwasawa type.

The set of nontrivial roots is 
\[ \Delta = \{ \alpha_1, \alpha_2, \alpha_1+\alpha_2, \delta, \delta \pm \alpha_1, \delta \pm \alpha_2, \delta \pm (\alpha_1+\alpha_2)\},\] where 
\begin{align*} 
 \delta(a_0 \bfD + a_1 \bfH_1 + a_2 \bfH_2) &= a_0,  \\
 \alpha_1(a_0 \bfD + a_1 \bfH_1 + a_2 \bfH_2) &=  2a_1 - a_2, \, \text{and}\\   
 \alpha_2(a_0 \bfD + a_1 \bfH_1 + a_2 \bfH_2) &=  - a_1 + 2a_2. 
\end{align*}
The root space decomposition of $\frakn$ determined by $\frakh$ is 
$ \frakn =  \bigoplus_{\alpha \in \Delta} \frakg_\alpha.$
The  root spaces 
\[\frakg_{\alpha_1} = \myspan \{ 1 \otimes E_{12}\},\quad \frakg_{\alpha_2} = \myspan \{ 1 \otimes E_{23}\}, \quad \frakg_{\alpha_1+\alpha_2} = \myspan \{1 \otimes E_{13}\}. \] 
\[  \frakg_{\delta + \alpha_1} = \myspan \{t \otimes E_{12}\}, \quad \frakg_{\delta + \alpha_2} = \myspan \{t \otimes E_{23}\}, \quad \frakg_{\delta +\alpha_1 + \alpha_2} = \myspan \{t \otimes E_{13}\},\]
\[  \frakg_{\delta -\alpha_1}= \myspan \{t \otimes E_{21}\}, \quad \frakg_{\delta-\alpha_2} = \myspan \{t \otimes E_{32}\}, \quad \frakg_{\delta-\alpha_1-\alpha_2} = \myspan \{t \otimes E_{31}\},\]
are one-dimensional and there is a single two-dimensional  root space 
  $\frakg_{\delta}= \myspan \{t \otimes
  H_{12}, t \otimes H_{23} \}.$

It will be convenient to let  
$\alpha_0 = \delta - \alpha_1 -\alpha_2$.
  The root vectors  for  the special roots $\alpha_0, \alpha_1, \alpha_2$ and $\delta$ are  $\bfH_{\alpha_0} = \tfrac{9}{16}\bfD - \tfrac 12\bfH_1 - \tfrac 12\bfH_2$,  
  $\bfH_{\alpha_1}=\tfrac12 \bfH_1$, $\bfH_{\alpha_2}=\tfrac12 \bfH_2$, and $\bfH_\delta = \tfrac{9}{16}\bfD$.
 The root vectors for all other roots can be obtained from sums of these.

An orthonormal basis for $\frakh$ is
\[  \calC_1 =  \left\{\smallfrac{3}{4} \bfD, \smallfrac{1}{2} \bfH_1, \smallfrac{1}{2\sqrt{3}}(\bfH_1 + 2\bfH_2)
   \right\}\] and  an orthonormal basis for $\n$ is 
\begin{multline*}
\calC_2 = \{ 1\otimes E_{12}, 1\otimes E_{23}, 1\otimes E_{13}\} \, \cup 
          \{  t\otimes E_{31}, t \otimes E_{21}, t \otimes E_{32}, \\t \otimes \tfrac{1}{\sqrt 2} H_{12},\, t \otimes \tfrac{1}{\sqrt{6}}(H_{12}+2H_{23}),  t \otimes E_{12}, t \otimes E_{23}, t \otimes E_{13} \}. 
\end{multline*}  
Hence, $\calC=\calC_1 \cup \calC_2$ is an orthonormal basis for $\fraks.$
It may be shown that the mean curvature vector for $\fraks$ is
$\bfH=\tfrac 92{\bfD} + \bfH_1 + \bfH_2$
and the endomorphism $\ad_{\bfH}$ is diagonal when represented with respect to the ordered basis $\calC_2$, 
\[ [\ad_{\bfH}|_\frakn]_{\calC_2} = \diag\left(1, 1, 2, \smallfrac52, \smallfrac72, \smallfrac72, \frac92, \frac92, \smallfrac{11}{2}, \smallfrac{11}{2}, \smallfrac{13}{2}\right). \]
Computing the Ricci endomorphism  using Equation \eqref{eqn: ricci-nilpotent}, we get
\begin{equation}\label{eqn: Ric n}
[\Ric^\n]_{\calC_2} = \diag\left(-\smallfrac 72, -\smallfrac 72,-\smallfrac 52,-2,-1,-1,0,0,1,1,2\right).
\end{equation}

We show that $(\fraks,\ip)$ is Einstein using Theorem \ref{thm: 3.9 conti-rossi}. Because
\[ [\Ric^\n -\ad_{\bfH}]_{\calC_2} = 
\diag\left(-\tfrac 92, -\tfrac 92,-\tfrac 92,-\tfrac 92, -\tfrac 92,-\tfrac 92,-\tfrac 92,-\tfrac 92  \right),\]
 $\Ric^\frakn = -\smallfrac{9}{2}\Id^\frakn + \ad_{\bfH}$. 
To  show  that 
 $\la \ad_A, \ad_{A'} \ra^{\End(\n)}_{\trace} = -\frac92 \la A, A' \ra$ for all $A, A' \in \frakh$, we compute both sides of this equality, obtaining that  with respect to the basis  $\{ \bfD, \bfH_1, \bfH_2 \}$ for $\frakh$, both sides of the equality are represented by the matrix  in Equation \eqref{eqn: mtx def}.
By Theorem \ref{thm: 3.9 conti-rossi}, the metric Lie algebra $(\fraks,\ip)$ is Einstein with Einstein constant $-\frac92$.  

Let  $\Lambda = \{\alpha_0, \alpha_1, \alpha_2\}$. 
 In order for a subset of $\Lambda$ to define an attached subalgebra of $(\fraks,\ip)$, the set $\Lambda$  must satisfy the 
hypotheses stated in first paragraph of Section \ref{subsection: definitions}.  This is easily verified.

The roots $\alpha_0, \alpha_1, \alpha_2$ in $\Lambda$ have corresponding root vectors $\bfH_{\alpha_0}, \bfH_{\alpha_1}, \bfH_{\alpha_2}$.
A short computation shows that reflections $s_0, s_1, s_2 : \frakh \to \frakh$ in these  are given by  
\begin{align} s_0(\bfH_{\alpha_0})&=-\bfH_{\alpha_0}, \quad &s_0(\bfH_{\alpha_1}) &=\bfH_{\alpha_1} +\tfrac{16}{25}\bfH_{\alpha_0}, \quad &s_0(\bfH_{\alpha_2}) &= \bfH_{\alpha_2} +\tfrac{16}{25}\bfH_{\alpha_0} \notag \\
s_1(\bfH_{\alpha_0})&=\bfH_{\alpha_0} + \bfH_{\alpha_1}, \quad &s_1(\bfH_{\alpha_1}) &=-\bfH_{\alpha_1}, \quad &s_1(\bfH_{\alpha_2}) &= \bfH_{\alpha_1} +\bfH_{\alpha_2} \notag \\
s_2(\bfH_{\alpha_0})&= \bfH_{\alpha_0} +\bfH_{\alpha_2}, \quad &s_2(\bfH_{\alpha_1}) &=\bfH_{\alpha_1} +\bfH_{\alpha_2}, \quad &s_2(\bfH_{\alpha_2}) &= -\bfH_{\alpha_2}. \notag 
\end{align}

Using that the set $\Lambda$ is a basis for $\frakh^*$, we define the  dual basis  $\calB = \{\bfB_0, \bfB_1, \bfB_2\}$ for $\frakh$,  
where  
\[ \bfB_0=\bfD, ~\bfB_1=\tfrac13 (2\bfH_1 + \bfH_2)   +\bfD, \, \text{and} \quad ~\bfB_2 =\tfrac13 (\bfH_1 + 2\bfH_2)+\bfD.\]
 With respect to this basis, our mean curvature vector is 
$\bfH= \tfrac52 \bfB_0 +\bfB_1 +\bfB_2.$

  \subsection{The definition  of the attached subalgebra}
We will use the subset 
  $\Lambda' = \{\alpha_2\}$ of $\Lambda$ to define an attached subalgebra. 
  Applying Definition \ref{defn: attached}, the attached subalgebra $\s'=\a' \oplus \n'$ defined by $\Lambda'$ has 
\begin{align*}\frakh' &= \myspan\{\bfB_0, \bfB_1\}, \quad \text{and} \\
  \n' &=  \myspan\{1\otimes E_{12}, 1\otimes E_{13}\} \\
                  &~\oplus \myspan\{t\otimes E_{31}, t\otimes E_{21}, t\otimes E_{32}, t\otimes H_{12}, t\otimes H_{23}, t\otimes E_{12}, t\otimes E_{23}, t\otimes E_{13}\}.\end{align*}
             An orthonormal basis for $\frakn'$ with respect to the restricted metric is 
\begin{align}
\label{ordered-basis-n-prime}
\calC_2' = \{1\otimes E_{12},  &1\otimes E_{13}, 
t\otimes E_{31}, t\otimes E_{21}, t\otimes E_{32}, \\
&\smallfrac{1}{\sqrt 2}t\otimes H_{12},\smallfrac{1}{\sqrt 6}t\otimes(H_{12}+2 H_{23}), t\otimes E_{12}, t\otimes E_{23}, t\otimes E_{13}\}. \notag
\end{align}
The vector $\bfZ$ is $\bfZ= \bfB_{0} + \bfB_{1}$, and $\alpha(\bfZ)>0$ for all roots in $\Delta$ except for $\alpha_2.$
\subsection{Properties of the attached subalgebra}
  To meet the definition of attached, we need to show that the  reflection  $\check{s}_{\alpha_2}$    permutes the 
  roots $\alpha$ for which 
$\alpha(\bfZ)>0$, and  that 
$\dim(\n_\alpha) = \dim(\n_{\check{s}_{\alpha_j}(\alpha)})$ for the roots in 
\[ \Delta \setminus \{\alpha_2 \} =  \{ \alpha_1,  \alpha_1+\alpha_2, \delta, \delta \pm \alpha_1, \delta \pm \alpha_2, \delta \pm (\alpha_1+\alpha_2)\}. \]
It is easy to confirm that $\check{s}_{\alpha_2}$    permutes the elements of the set $ \Delta \setminus \{\alpha_2 \} $ using
\[ \check{s}_{\alpha_2}(c_0 \delta + c_1 \alpha_1 + c_2 \alpha_2)= c_0 \delta +c_1 (\alpha_1 + \alpha_2)  -c_2 \alpha_2.\]
All of the root spaces are one-dimensional except for $\frakn_\delta$, and $\delta$  is fixed by  $\check{s}_{\alpha_2}$, so 
$\dim(\n_\alpha) = \dim(\n_{\check{s}_{\alpha_j}(\alpha)})$ for the roots under consideration.

The mean curvature vector for  $\n'$ in $\s'$ is 
$\bfH' = \tfrac92\bfD + \bfH_1 + \tfrac12 \bfH_2.$

The restriction of 
$\ad_{\bfH'}$ to $\n'$ is given by
\begin{equation}\label{eqn: adH'} 
[\ad_{\bfH'}|_{\frakn'}]_{\calC_2'} = 
\diag\left(\smallfrac32, \smallfrac32, 3, 3, \smallfrac92, \smallfrac92, \smallfrac92, 6, \smallfrac92, 6\right),\end{equation}
while the restriction of 
$\ad_{\bfH}$ to $\n'$ is
\begin{equation}\label{eqn: adH} 
[\ad_{\bfH}|_{\frakn'}]_{\calC'_2} = \diag\left(1, 2, \smallfrac52, \smallfrac72, \smallfrac72, \smallfrac92, \smallfrac92, \smallfrac{11}{2}, \smallfrac{11}{2}, \smallfrac{13}{2}\right).  \end{equation}
The  Ricci endomorphism for  $\n'$ is
\begin{equation}\label{eqn: Ric nprime}
[\Ric^{\n'}]_{\calC'_2} = \diag\left(-3,-3,-\smallfrac 32,-\smallfrac 32, 0,0,0,\frac 32, 0, \smallfrac 32\right).
\end{equation}

The hypotheses of the Main Theorem are satisfied. 
The Main Theorem says the following are equivalent:
   \begin{enumerate}
       \item 
 ${\displaystyle \Ric^{\s}(X) = \Ric^{\s'}(X)}$   for any $X \in \s'$.
\item $\Ric^\n (X) - \Ric^{\n'} (X) =  [H - H', X]$ for any $X \in \n'$.
  \item The Jacobi Star Condition  holds. 
\end{enumerate}
 Part (i) is verified by
\[ [\Ric^\fraks|_{\fraks'}]_{\calC'} = 
\diag(-\smallfrac92,-\smallfrac92,-\smallfrac92,-\smallfrac92,-\smallfrac92,-\smallfrac92,-\smallfrac92,-\smallfrac92,-\smallfrac92,-\smallfrac92) = [\Ric^{\fraks'}]_{\calC'} .\]
Equations \eqref{eqn: adH'} and \eqref{eqn: adH}  tell us that
the difference $\ad_{H -H'}|_{\n'}$ is given by
\begin{equation}\label{eqn: MCV-diff}[(\ad_{H -H'})|_{\n'}]_{\calC'_2} = \diag\left(-\smallfrac12, \smallfrac12, -\smallfrac12, \smallfrac12, -1, 0, 0, -\smallfrac12, 1, \smallfrac12\right).
\end{equation}
Taking the difference of Equations \eqref{eqn: Ric nprime} and \eqref{eqn: Ric n}, we get
\begin{equation}\label{eqn: conclusion}[\Ric^\frakn|_{\frakn'}]_{\calC_2'} - [\Ric^{\frakn'}]_{\calC_2} 
= \diag\left(-\smallfrac 12, \smallfrac 12, -\smallfrac 12, \smallfrac 12, -1, 0, 0, -\smallfrac 12, 1, \smallfrac 12 \right).
\end{equation}
Comparing Equations \eqref{eqn: conclusion} and \eqref{eqn: MCV-diff} shows that 
\[ 
[\Ric^\frakn|_{\frakn'}]_{\calC_2'} - [\Ric^{\frakn'}]_{\calC_2} 
= [\ad_{H -H'}|_{\n'}]_{\calC_2'} .\]
Thus Part (ii) holds.    
Once (i) or (ii) holds, of course (iii) must hold by the Main Theorem.

By Theorem \ref{thm: minimal}, the attached submanifold defined by
$(\fraks', \ip')$ is minimal in the solvmanifold corresponding to 
$(\fraks, \ip)$.
To verify that the submanifold is not totally geodesic, we show the sets 
$\Lambda'=\{\alpha_2\}$ and $\Lambda \setminus \Lambda' = \{\alpha_0,
\alpha_1\}$ are not orthogonal.  Because $\la \alpha_1,\alpha_2\ra=-2$, they
are not, so by Proposition \ref{prop: tot-geodesic}, the submanifold is not
totally geodesic.

In this example, the nilpotent Lie algebra $\frakn$ has dimension $11$
and has five steps.   There are ten nontrivial root spaces for the action of  $\frakh$ on $\frakn$, and ten of these  
root  spaces are one-dimensional and one is two-dimensional.

Suppose that $\frakn$ is the nilpotent part of the Iwasawa decomposition 
of a semisimple Lie group $\frakg$ of noncompact type. Because the dimension of the center of $\frakn$ is one,  
$\frakg$  must be simple. 
The  subspace $\myspan \{ 1 \otimes E_{12}, 1 \otimes E_{23}, 
t \otimes E_{31}  \}$ generates $\frakn$ and 
is the sum of two $\ad_{\bfH}$ eigenspaces, so any 
extension of $\myspan \{ \bfH \}$  to a Cartan-subalgebra in $\frakg$ has dimension two or three. Hence $\frakg$ has rank $1$, $2$ or $3.$ 
It can't be rank one, because $\frakn$  is not two step.  If $\frakg$
were rank two, the corresponding symmetric space would be 
$13$-dimensional. If $\frakg$ were rank three then the 
corresponding symmetric space would be $14$-dimensional.  
Comparing with \cite{helgason-78}, Tables V and VI,
or Besse \cite{besse-87}, Tables 7.104 and 7.105, we deduce that 
there are  symmetric spaces of noncompact type with 
dimension $13$ other than hyperbolic space, and there no rank three
symmetric spaces of dimension 14. Thus, $\frakn$ is not the nilpotent Lie algebra in the Iwasawa decomposition of any real  semisimple Lie algebra of noncompact type.   

In conclusion, we have shown that $(\fraks, \ip)$ is an Einstein metric Lie
algebra, and that $(\fraks', \ip')$ is a metric
subalgebra  whose 
Ricci endomorphism is the restriction of the Ricci endomorphism for
$(\fraks, \ip)$.  Hence the metric 
subalgebra is also Einstein with  the same
Einstein constant. Furthermore, this metric subalgebra is minimal but 
not totally geodesic.  Finally, we  showed that the associated $\frakn$ 
is not the nilpotent Lie algebra from a symmetric space of noncompact type.

\bibliographystyle{amsalpha}
\bibliography{bibfile}
 
 \end{document}